\documentclass{article}

\usepackage{latexsym}
\usepackage{amsmath}
\usepackage{amssymb,amsfonts,amscd}
\usepackage{graphicx}
\usepackage{multirow}
\usepackage{hyperref}

\newtheorem{theorem}{Theorem}

\newtheorem{proposition}[theorem]{Proposition}

\newtheorem{corollary}[theorem]{Corollary}
\newtheorem{lemma}[theorem]{Lemma}
\newtheorem{rem}[theorem]{Remark}
\newtheorem{defn}[theorem]{Definition}
\newtheorem{ex}[theorem]{Example}
\newtheorem{assumption}[theorem]{Assumption}

\newenvironment{definition}{\begin{defn}\em}{\end{defn}}
\newenvironment{example}{\begin{ex}\em}{\end{ex}}
\newenvironment{proof}{{\noindent\bf Proof.\ }}{\qed}
\newenvironment{proofThmMain}{{\noindent\bf Proof of Theorem~\ref{thm:MainThm}.\ }}{\qed}
\newenvironment{proofThmMainV}{{\noindent\bf Proof of Theorem~\ref{thm:MainThmV}.\ }}{\qed}

\newcommand{\bN}{{\mathbb N}}
\newcommand{\bC}{{\mathbb C}}
\newcommand{\bP}{{\mathbb P}}

\newcommand{\bR}{{\mathbb R}}

\newcommand{\bZ}{{\mathbb Z}}

\newcommand\qed{{\hspace*{\fill}$\Box$\vskip12pt plus 1pt}}

\newcommand\sF{{\mathcal F}}
\newcommand\sG{{\mathcal G}}

\newcommand\sL{{\mathcal L}}
\newcommand\sM{{\mathcal M}}

\newcommand\sU{{\mathcal U}}
\newcommand\sV{{\mathcal V}}
\newcommand\sW{{\mathcal W}}

\newcommand\Var{\sV}
\newcommand\Sing{{\rm Sing}}
\newcommand{\grad}{{\nabla}}
\newcommand{\Real}{{\bf RealPoints}}
\newcommand{\RealComp}{{\bf RealPointsComponent}}

\def\rank{{\mathop{\rm rank~}\nolimits}}
\def\Null{{\mathop{\rm null~}\nolimits}}
\newcommand{\param}{{\rm param}}
\renewcommand{\prod}{{\rm prod}}

\begin{document}

\title{Numerically computing real points on algebraic sets}

\author{Jonathan D. Hauenstein\thanks{Department of Mathematics, Mailstop 3368,
Texas A\&M University, College Station, TX 77843 (jhauenst@math.tamu.edu,
{\tt www.math.tamu.edu/$\sim$jhauenst}).
This author was supported by Texas A\&M University,
Institut Mittag-Leffler (Djursholm, Sweden), NSF grants DMS-0915211
and DMS-1114336, and DOE ASCR grant DE-SC0002505.}}

\date{February 25, 2012}

\maketitle

\begin{abstract}

\noindent Given a polynomial system $f$, a fundamental question
is to determine if $f$ has real roots.  Many algorithms involving
the use of infinitesimal deformations have been proposed to answer this question.
In this article, we transform an approach of Rouillier, Roy, and Safey El Din,
which is based on a classical optimization approach of Seidenberg, to develop a homotopy
based approach for computing at least one point on each connected component
of a real algebraic set.  Examples are presented demonstrating the effectiveness of
this parallelizable homotopy based approach.


\end{abstract}
\small
\normalsize
\section{Introduction}

Computing real roots of a polynomial system is a difficult and extremely important
problem.  In many applications in science, engineering, and economics, the real roots
are the only ones of interest.  Due to the importance of this problem,
many approaches have been proposed.  Two approaches are the
cylindrical algebraic decomposition algorithm \cite{Collins}
and so-called critical point methods, such as 
Seidenberg's approach of computing critical points of the distance function \cite{Seidenberg}.
The cylindrical algebraic decomposition algorithm
has doubly exponential complexity in the number of variables.
However, using the idea of Seidenberg and related ideas developed
in \cite{BPR,Canny,GV88,GV92,HRS,R92}, algorithms with asymptotically optimal complexity estimates
for computing at least one real point on each connected component of a real algebraic set
were developed.  Other related approaches for computing real roots
are presented in \cite{RealSolving,BGHM,BGHSS,Real} and the references therein.
The approach presented here will transform the algorithms presented in \cite{RealSolving,Real}
into a homotopy based algorithm.

Several homotopy based algorithms have been proposed to compute real roots of a polynomial system.
The algorithms in \cite{RealCurve} and \cite{RealSurface} utilize critical point methods
to decompose the real points of a complex curve and a complex surface with finitely many
singularities, respectively.  An algorithm for directly computing only the
real roots that are isolated over the complex numbers is presented in \cite{Fewnomials}.
The complexity of this approach depends upon the fewnomial structure of the given polynomial system.
The approach presented below is not restricted to low-dimensional cases and the real roots
are not assumed to be isolated over the complex numbers.

Two other nonhomotopy based algorithms are presented in \cite{RealRadical} and \cite{RealPositive}.
The approach in \cite{RealRadical} (see also \cite{NumRealRadicalSummary})
uses semidefinite programming for computing real roots.
This algorithm computes every real root assuming the number of real roots is finite.
The approach in \cite{RealPositive} uses tools related to
maximum likelihood estimation in statistics for computing real positive roots
of certain types of polynomial systems.

The rest of the article is structured as follows.
The remainder of this section describes the needed concepts from
complex, real, and numerical algebraic geometry and a brief introduction to Puiseux series.
Section~\ref{Sec:Real} describes the homotopy based approach
with examples demonstrating the algorithm in Section~\ref{Sec:Examples}.

\subsection{Algebraic sets and genericity}\label{Sec:AlgebraicSets}

Let $f:\bC^N\rightarrow\bC^n$ be a polynomial system and $\Var(f) = \{x\in\bC^N~|~f(x) = 0\}$.
The set $\Var(f)\subset\bC^N$ is called the {\em algebraic set associated to $f$}.
A set $X\subset\bC^N$ is called an {\em algebraic set} if there exists a
polynomial system $g:\bC^N\rightarrow\bC^m$ such that $X = \Var(g)$.
An algebraic set $X\subset\bC^N$ is {\em reducible} if there exists
algebraic sets $Y,Z\subset\bC^N$, which are proper subsets of $X$,
such that $X = Y \cup Z$.  An algebraic set is {\em irreducible}
if it is not reducible.  For an irreducible algebraic set $X$, the subset of
manifold points $X_{\rm reg}$ is dense in $X$, open, and connected.  The {\em dimension} of an irreducible
algebraic set $X$ is the dimension of $X_{\rm reg}$ as a complex manifold.

On irreducible algebraic sets, we can define the notion of genericity.

\begin{definition}\label{Def:Genericity}
Let $X\subset\bC^N$ be an irreducible algebraic set.  A property $P$ is said
to hold {\em generically} on $X$ if the subset of points in $X$ which
do not satisfy $P$ are contained in a proper algebraic subset of $X$.
That is, there is a nonempty Zariski open subset $U$ of $X$ such that $P$ holds
at every point in $U$.  Each point in $U$ is called a {\em generic point} of $X$
with respect to $P$.
\end{definition}

Since every proper algebraic subset of $\bC$ is a finite set, a property $P$ holds
generically on $\bC$ if $P$ holds at all but finitely many points in $\bC$.

Every algebraic set $X$ can be written uniquely (up to reordering)
as the finite union of inclusion maximal irreducible algebraic sets,
called the {\em irreducible decomposition} of $X$.  That is,
there are irreducible algebraic sets $A_1,\dots,A_k$ such that
$$X = \bigcup_{i=1}^k A_i \hbox{~~and~~} A_i \not\subset A_j \hbox{~for~} i\neq j.$$
Each $A_i$ is called an {\em irreducible component} of $X$.

The dimension of an algebraic set is the maximum dimension of its irreducible components.
An algebraic set is called {\em pure-dimensional} if each irreducible component
has the same dimension.  The {\em pure $i$-dimensional component} of an algebraic set
is the union of the irreducible components of dimension $i$.  In summary, the
algebraic set $\Var(f)$ has an irreducible decomposition of the form
\begin{equation}\label{Eq:IrredDecomp}
\Var(f) = \bigcup_{i=0}^{\dim\Var(f)} V_i = \bigcup_{i=0}^{\dim\Var(f)} \bigcup_{j=1}^{k_i} V_{i,j}
\end{equation}
where $V_i$ is the pure $i$-dimensional component of $\Var(f)$ and
each $V_{i,j}$ is a distinct $i$-dimensional irreducible component.

\subsection{Decomposition of real algebraic sets}\label{Sec:ConnectedComponents}

A {\em real algebraic set} are subsets of $\bR^N$ which arise as the intersection of algebraic
sets in $\bC^N$ with $\bR^N$.  That is, a set $X\subset\bR^N$ is a real algebraic set if
there is an algebraic set $Y\subset\bC^N$ such that $X = Y \cap\bR^N$.
For a polynomial system $f:\bR^N\rightarrow\bR^n$, the {\em real algebraic set associated to $f$} is
$\Var_\bR(f) = \Var(f)\cap\bR^N = \{x\in\bR^N~|~f(x) = 0\}$.

Consider the algebraic set $X = \Var(y^2 - x^2(x-1))\subset\bC^2$.  It is easy to verify
that $X$ is an irreducible algebraic set and, hence, both $X$ and $X_{\rm reg}$ are connected.
However, the real algebraic set $X\cap\bR^2$ is not connected.
This example suggests that we should consider decomposing
real algebraic sets into connected components.

A real algebraic set $X\subset\bR^N$ can be written uniquely (up to reordering) as the
disjoint union of finitely many path-connected sets $C_1,\dots,C_\ell\subset\bR^N$
such that $C_i$ and $V\setminus C_i$ are both closed in the Euclidean topology on $\bR^N$.
Each $C_i$ is called a {\em connected component} of $X$ and one can verify that
it is a semi-algebraic set.  Expanded details regarding real algebraic sets and decompositions can be found in \cite{Basu,BCR}.

To demonstrate this decomposition and contrast it with the irreducible decomposition of algebraic sets,
consider the algebraic sets $X = \Var(y^2 - x^2(x-1))$, $Y = \Var(x-y)$,
and $Z = X \cup Y$ with corresponding real algebraic sets $X_\bR = X\cap\bR^2$, $Y_\bR = Y\cap\bR^2$, and $Z_\bR = Z\cap\bR^2$.
It is easy to verify that $X$ and $Y$ are irreducible algebraic sets with $Z$ clearly being a reducible algebraic set.
The set $X_\bR$ consists of two connected components, namely $C_1 = \{(0,0)\}$ and a connected curve $C_2 = Y_\bR \setminus C_1$.
Since the real algebraic sets $Y_\bR$ and $Z_\bR$ are connected, $Y_\bR$ and $Z_\bR$ each have only one connected component.

\subsection{Puiseux series}\label{Sec:Puiseux}

Since we will utilize Puiseux series in Section~\ref{Sec:Real},
we will provide a brief review here.
For more detailed information, see \cite{Basu}.

The field of algebraic Puiseux series over $\bC$ is
$$\bC\langle\epsilon\rangle = \left.\left\{\sum_{j\geq j_0} a_j \epsilon^{j/q}~\right|~j_0\in\bZ, q\in\bN, a_j\in\bC
\hbox{~with~}a_{j_0}\neq 0 \right\}.$$
To simplify the notation, we shall define $a_j = 0$ for all $j < j_0$.
An element in $\bC\langle\epsilon\rangle$ is {\em bounded} if $j_0 \geq 0$
and {\em infinitesimal} if $j_0 > 0$.
The subset consisting of bounded elements, denoted $\bC_b\langle\epsilon\rangle$,
is a ring which is naturally mapped to $\bC$ by the ring homomorphism $\lim_0$ defined by
$$\hbox{$\lim_0$} \sum_{j\geq j_0} a_j \epsilon^{j/q} = a_0.$$

\subsection{Numerical irreducible decomposition and witness sets}\label{Sec:WitnessSets}

Let $f:\bC^N\rightarrow\bC^n$ be a polynomial system.
A numerical irreducible decomposition of $\Var(f)$, first presented in \cite{NAG},
is a numerical decomposition analogous to (\ref{Eq:IrredDecomp}) using witness sets
(see \cite[Chaps. 12-15]{SW05} for more expanded details).
Suppose that $V$ is the pure $i$-dimensional component of $\Var(f)$ with $d = \deg V$.
For a fixed generic $i$-codimensional linear space $H\subset\bC^N$, we have that
$V\cap H$ consists of $d$ points.  Let $L:\bC^N\rightarrow\bC^i$ be a system of linear polynomials such
that $\Var(L) = H$.  The set $V\cap\Var(L) = V\cap H$ is called a
{\em witness point set} for $V$ with the triple
$\sW = \{f,L,V\cap\Var(L)\}$ called a {\em witness set} for $V$.
A {\em numerical irreducible decomposition} of $\Var(f)$ is of the form
\begin{equation}\label{Eq:NumIrredDecomp}
\bigcup_{i=0}^{\dim\Var(f)} \bigcup_{j=1}^{k_i} \sW_{i,j}
\end{equation}
where $\sW_{i,j}$ is a witness set for a distinct $i$-dimensional irreducible component of $\Var(f)$.
We note that the union of witness sets in (\ref{Eq:NumIrredDecomp}) should be considered
as a formal union.  Numerical irreducible decompositions can be computed using the algorithms
presented in \cite{LocalDim,RegenCascade,Cascade,SVW01a,SVW01b,SVW02b,SVW01c,NAG}.

\subsection{Trackable paths}\label{Sec:Trackable}

Numerical homotopy methods rely on the ability to construct
homotopies with solution paths that are {\em trackable}.  The
following is the definition of a trackable solution path starting
at a nonsingular point adapted from \cite{Regen}.

\begin{definition}\label{def:trackable}
Let $H(x,t):\bC^N\times\bC\rightarrow\bC^N$ be polynomial in $x$
and complex analytic in $t$ and let $x^*$ be a nonsingular isolated solution
of $H(x,1) = 0$.  We say that $x^*$ is {\em trackable} for $t\in(0,1]$ from $t = 1$
to $t = 0$ using $H(x,t)$ if there is a smooth map $\xi_{x^*}:(0,1]\rightarrow\bC^N$ such that
$\xi_{x^*}(1) = x^*$ and, for $t\in(0,1]$, $\xi_{x^*}(t)$ is a nonsingular isolated solution of $H(x,t) = 0$.
\end{definition}

The solution path starting at $x^*$ is said to {\em converge} if
$\lim_{t\rightarrow 0^+} \xi_{x^*}(t)\in\bC^N$, where
$\lim_{t\rightarrow 0^+} \xi_{x^*}(t)$ is called the {\em endpoint} (or {\em limit point}) of the path.

\section{Real points on an algebraic set}\label{Sec:Real}

Let $f:\bR^N\rightarrow\bR^n$ be a polynomial system and $V\subset\Var_\bC(f)$ be a
pure $d$-dimensional algebraic set.  The main problem we consider is, given a witness set
$\{f,\sL,W\}$ for $V$, compute a finite set of points which contains
at least one point on each connected component of $\Var_\bR(f)$ contained
in $V$.  We note that if $d = 0$, then $V = W$ so that one can compute
the real points in $V$ simply by considering the
finitely many points in $W$.  Hence, we will assume that $d > 0$.

We will also reduce to the case $n = N - d$.  One way to
always reduce down to this case is to consider the polynomial $g$ which
is the sum of squares of $f$, that is, $g = f_1^2 + \cdots + f_n^2$,
with $V_g = \Var(g)$.  If $T$ is a finite set of points which contains at
least one point on each connected component of $\Var_\bR(g)$,
then $T\cap V$ contains at least one point on each connected
component of $\Var_\bR(f) = \Var_\bR(g)$ contained in $V$.
The set $T\cap V$ can be computed from $T$ and
a witness set for $V$ using the homotopy membership test \cite{SVW01b}.

We summarize the assumptions in the following statement.

\begin{assumption}\label{assume:Basic}
Let $N > d > 0$, $f:\bR^N\rightarrow\bR^{N-d}$ be a polynomial system, and
$V\subset\Var(f)$ be a pure $d$-dimensional algebraic set with witness set $\{f,\sL,W\}$.
\end{assumption}

The following lemma considers the solutions of $f(x) = z$ for $z\in\bC^{N-d}$.

\begin{lemma}\label{lemma:PerturbF}
With Assumption~\ref{assume:Basic}, there is a nonempty Zariski open set $Z\subset\bC^{N-d}$
such that, for every $z\in Z$, $\Var(f - z)$ is a smooth algebraic set of dimension $d$.
\end{lemma}
\begin{proof}
Let $r = \dim \overline{f(\bC^N)}$ and $c = N - r$, which are called
{\em rank of $f$} and the {\em corank of $f$} respectively \cite[\S 13.4]{SW05}.
Since $\overline{f(\bC^N)}\subset\bC^{N-d}$, we have $r \leq N-d$ and hence $d \leq N - r = c$.
Since $\Var(f)$ has a component of dimension $d$, Theorem~13.4.2 of \cite{SW05}
yields that $d\geq c$.  Therefore, $c = d$ and $r = N - d$.
The lemma now follows immediately from Lemma~13.4.1 of \cite{SW05}.
\end{proof}

Lemma~\ref{lemma:PerturbF} permits the use of continuation techniques as
stated in the following theorem.

\begin{theorem}\label{thm:MainThmV}
Suppose that Assumption~\ref{assume:Basic} holds.
Let $z\in\bR^{N-d}$, $\gamma\in\bC$, $y\in\bR^N\setminus\Var_\bR(f)$, $\alpha\in\bC^{N-d+1}$, and
$H:\bC^N\times\bC^{N-d+1}\times\bC\rightarrow\bC^{2N-d+1}$ be the homotopy defined by
\begin{equation}\label{eq:HomotopyV}
H(x,\lambda,t) = \left[\begin{array}{c}
 f(x) - t \gamma z \\ \lambda_0(x-y) + \lambda_1 \grad f_1(x)^T + \cdots + \lambda_{N-d} \grad f_{N-d}(x)^T \\
 \alpha_0 \lambda_0 + \cdots + \alpha_{N-d} \lambda_{N-d} - 1 \end{array}\right]
 \end{equation}
where $f(x) = [f_1(x),\dots,f_{N-d}(x)]^T$ such that following statements hold.
\begin{enumerate}
\item\label{item:Main1V} The set $S\subset\bC^N\times\bC^{N-d+1}$ of roots of $H(x,\lambda,1)$ is finite and
    each is a nonsingular solution of $H(x,\lambda,1) = 0$.
\item\label{item:Main2V} The number of points in $S$ is equal to the maximum number of isolated solutions of $H(x,\lambda,1) = 0$ as
     $z$, $\gamma$, $y$, and $\alpha$ vary over sets $\bC^{N-d}$, $\bC$, $\bC^N$, and $\bC^{N-d+1}$, respectively.
\item\label{item:Main3V} The solution paths defined by $H$ starting, with $t = 1$, at the points in $S$
    are trackable.
\item\label{item:Main4V} If $\pi(x,\lambda) = x$,
    $$\begin{array}{lcll}
    E &=& \left\{\left.\lim_{t\rightarrow0^+} \xi_s(t)~\right|~s\in S \hbox{~and the solution path~}\xi_s \hbox{~converges}\right\}, & \hbox{and} \\
    E_1 &=& \left\{\left.\lim_{t\rightarrow0^+} \pi(\xi_s(t))~\right|~s\in S \hbox{~and the path~}\pi(\xi_s) \hbox{~converges}\right\}, & \end{array}$$
    we have $E_1 = \pi(E)$.
\end{enumerate}
Then, $E_1\cap V\cap\bR^N$ contains a point on each connected component of $\Var_\bR(f)$ contained in $V$.
\end{theorem}

The homotopy $H$ defined in (\ref{eq:HomotopyV}) is based
on the classical approach of Seidenberg \cite{Seidenberg}.
If $y\in\bR^N\setminus\Var_\bR(f)$, consider the quadratic polynomial
$$d_y(x) = (x-y)^T (x-y) = \sum_{i=1}^N (x_i - y_i)^2$$
and the optimization problem
\begin{enumerate}
\item[(P)] \hspace{1.6in} $\min \left.\left\{d_y(x)~\right|~x\in\Var_\bR(f)\right\}$.
\end{enumerate}
We want to compute the points on $\Var(f)$ for which $\grad d_y(x) = 2(x-y)^T$ and $\grad g(x)$
are linearly dependent.  The approach in \cite{Real} for hypersurfaces
uses determinants to describe this linear dependence condition,
while the approach in Theorem~\ref{thm:MainThmV} uses auxiliary variables $\lambda$.
In particular, the polynomial system $\sG_{f,y}:\bC^N\times\bP^{N-d}\rightarrow\bC^{2N-d}$ defined by
\begin{equation}\label{eq:G}
G_{f,y}(x,\lambda) = \left[\begin{array}{c} f(x) \\
\lambda_0(x-y) + \lambda_1\grad f_1(x)^T + \cdots + \lambda_{N-d}\grad f_{N-d}(x)^T \end{array}\right]
\end{equation}
comprises the Fritz John conditions \cite{John} for problem (P)
and provides necessary conditions for optimality.
That is, if $\xi$ is a local minimizer for problem (P),
then there exists $\lambda\in\bP^{N-d}$ such that $(\xi,\lambda)\in\Var(\sG_{f,y})$.

Clearly, $x\in\Var(f)$ such that
$$\rank \left[\begin{array}{cccc} x-y & \grad f_1(x)^T & \cdots & \grad f_{N-d}(x)^T \end{array}\right] \leq N - d$$
if and only if there exists $\lambda\in\bP^{N-d}$ such that $(x,\lambda)\in\Var(\sG_{f,y})$.
A point $x\in\pi(\Var(\sG_{f,y}))$ is called a {\em critical point} of the distance function
with respect to $f$, where $\pi(x,\lambda) = x$.

Consider $\Sing(f) = \{x\in\bC^N~|~\rank Jf(x) < N - d\}$,
where $Jf(x)$ is the Jacobian matrix of $f$ evaluated at $x$.
If $\Sing(f)$ is positive dimensional, then $\Var(\sG_{f,y})$ is also positive dimensional.
By using Lemma~\ref{lemma:PerturbF}, we can consider smooth algebraic sets
thereby allowing the computation of finitely many points
in $\Var_\bR(f)$ containing the points of interest.

The following lemma will be used to complete the proof of Theorem~\ref{thm:MainThmV}

\begin{lemma}\label{lemma:KeyLemma}
Suppose that Assumption~\ref{assume:Basic} holds.
Let $\epsilon$ be an infinitesimal, $y\in\bR^N\setminus\Var_\bR(f)$,
$z\in\bR^{N-d}$ with $z_i\neq 0$, and $f_\epsilon(x) = f(x) - \epsilon z$
be such that $\Var(\sG_{f_\epsilon,y})$ is finite
and $|\Var(\sG_{f_\epsilon,y})|$ is equal to the maximum number of isolated solutions as
$y$ and $z$ varies over the sets $\bC^N$ and $\bC^{N-d}$, respectively.
Then,
\begin{enumerate}
\item\label{item:Key1} $V\subset\hbox{$\lim_0$}\left(\Var(f_\epsilon)\cap\bC_b\langle\epsilon\rangle^N\right)$,
\item\label{item:Key2} $\hbox{$\lim_0$}\left(\Var(f_\epsilon)\cap \bC_b\langle\epsilon\rangle^N\right)\cap\bR^N =
\hbox{$\lim_0$}\left(\Var(f_1^2 - \epsilon^2 z_1^2, \dots, f_{N-d}^2 - \epsilon^2 z_{N-d}^2)\cap\bR_b\langle\epsilon\rangle^N\right)$, and
\item\label{item:Key3} $\hbox{$\lim_0$}\left(\pi(\Var(\sG_{f_\epsilon,y}))\cap\bC_b\langle\epsilon\rangle^N\right)\cap\bR^N$
contains a point in each connected component of $\Var_\bR(f)$ contained in $V$ where $\pi(x,\lambda) = x$.
\end{enumerate}
\end{lemma}
\begin{proof}
This setup implies that $\Var(f_\epsilon)$ is a $d$-dimensional smooth algebraic set
for which we clearly have $V\subset\lim_0(\Var(f_\epsilon)\cap\bC_b\langle\epsilon\rangle^N)$
yielding Item~\ref{item:Key1}.
Item~\ref{item:Key2} follows from the fact that $z_i\in\bR$ and
$$\hbox{$\lim_0$}\left(\Var(g - \epsilon)\cap\bC_b\langle\epsilon\rangle^N\right)\cap\bR^N =
\hbox{$\lim_0$}\left(\Var(g^2 - \epsilon^2)\cap\bR_b\langle\epsilon\rangle^N\right)$$
for any polynomial $g:\bR^N\rightarrow\bR$.
Item~\ref{item:Key3} follows by using the same proof as Lemma~3.7 in \cite{Real} with
the replacement of Lemma~3.6 of \cite{Real} with Items~\ref{item:Key1} and~\ref{item:Key2}.
\end{proof}

Before we prove Theorem~\ref{thm:MainThmV}, we note
that the polynomial system $\sG_{g,y}(x,\lambda)$ defined in (\ref{eq:G}),
has $\lambda\in\bP^{N-d}$.  The polynomial system $H(x,\lambda,0)$ defined in
(\ref{eq:HomotopyV}) has $\lambda\in\bC^{N-d+1}$ restricted to the Euclidean patch
defined by $\alpha_0\lambda_0 + \cdots + \alpha_{N-d}\lambda_{N-d} = 1$.
Item~\ref{item:Main3V} in Theorem~\ref{thm:MainThmV} enforces that this Euclidean
patch is in general position with respect to the finitely many solution paths.
Therefore, we can use the results of Lemma~\ref{lemma:KeyLemma} in the following
proof of Theorem~\ref{thm:MainThmV}.

\bigskip

\begin{proofThmMainV}
Let $\epsilon$ be an infinitesimal and $f_\epsilon = f - \epsilon z$.
Item~\ref{item:Main2V} yields that $|S| = |\Var(\sG_{f_\epsilon,y})| < \infty$.
The result will follow from Lemma~\ref{lemma:KeyLemma} upon showing
\begin{equation}\label{eq:E1lim}
E_1 = \hbox{$\lim_0$}\left(\pi(\Var(\sG_{f_\epsilon,y}))\cap\bC_b\langle\epsilon\rangle^N\right).
\end{equation}
We will deduce (\ref{eq:E1lim}) by comparing the polynomial systems $\sG_{f_\epsilon,y}$ and
$$\sG^a_{f_\epsilon,y}(x,\lambda) = \left[\begin{array}{c} f(x) - \epsilon z \\
  \lambda_0(x-y) + \lambda_1\grad f_1(x)^T + \cdots + \lambda_{N-d}\grad f_{N-d}(x)^T \\
  \alpha_0 \lambda_0 + \cdots + \alpha_{N-d} \lambda_{N-d} - 1 \end{array}\right].$$
Item~\ref{item:Main2V} also yields that
$|S| = |\Var(\sG_{f_\epsilon,y})| = |\Var(\sG^a_{f_\epsilon,y})|$.
In particular, by abuse of notation regarding $\pi$, we have
\begin{equation}\label{eq:MG}
\pi(\Var(\sG_{f_\epsilon,y})) = \pi(\Var(\sG^a_{f_\epsilon,y})).
\end{equation}
Since there are finitely many homotopy paths,
there exists $0 < t_0 < 1$ such that all of the homotopy paths for $H$
with $0 < t < 2 t_0$ are described by the points in
$\Var(\sG^a_{f_\epsilon,y})\subset\bC\langle\epsilon\rangle^{2N-d}$
by replacing $\epsilon$ with $t \gamma$.
This yields that the set of limit points of the homotopy $H_0(x,\lambda,t) := H(x,\lambda,(1-t)\cdot t_0)$
starting at the roots of $H(x,\lambda,t_0)$ is
$$T = \hbox{$\lim_0$}\left(\Var(\sG^a_{f_\epsilon,y})\cap\bC_b\langle\epsilon\rangle^{N+2}\right).$$
Since, by Items~\ref{item:Main1V} and~\ref{item:Main3V}, the homotopy paths of $H$ are nonsingular for $t\in(0,1]$,
coefficient-parameter continuation \cite{CoeffParam} yields that $T = E$.
Item~\ref{item:Main4V} yields
$$\pi(E) = \pi\left(\hbox{$\lim_0$}\left(\Var(\sG^a_{f_\epsilon,y})\cap\bC_b\langle\epsilon\rangle^{N+2}\right)\right)
= \hbox{$\lim_0$}\left(\pi(\Var(\sG^a_{f_\epsilon,y}))\cap\bC_b\langle\epsilon\rangle^N\right) = E_1.$$
This equation together with (\ref{eq:MG}) yields (\ref{eq:E1lim}).
\end{proofThmMainV}

We note that in the hypersurface case, that is $n = N - d = 1$, if
$f$ has degree $2k$, the 2-homogeneous B\'ezout count yields that
$$|S| \leq K(N,2k):= N \cdot 2k \cdot (2k - 1)^{N-1}.$$
In particular, $\Var_\bR(f)$ can have at most $K(N,2k)$ connected components and
hence $K(N,2k)$ bounds the number of real roots of $f$ that are isolated over $\bR$.
This bound is only $N$ times larger than the bound obtained in \cite[Prop.~11.5.2]{BCR}.

\subsection{An algorithm}\label{Sec:Alg1}

Theorem~\ref{thm:MainThmV} yields an approach for computing a point on
each connected component of $\Var_\bR(f)$.  Before presenting
an algorithm which implements the ideas of this theorem, we state two remarks.
First, Item~\ref{item:Main2V} of Theorem~\ref{thm:MainThmV}
holds for a nonempty Zariski open set of $\bC^{N-d}\times\bC\times\bC^N\times\bC^{N-d+1}$.
The following algorithm assumes that the given point $(z,\gamma,y,\alpha)$
lies in this Zariski open set.  As part of the procedure, it computationally verifies
Items~\ref{item:Main1V},~\ref{item:Main3V}, and~\ref{item:Main4V} of Theorem~\ref{thm:MainThmV} hold.
Second, the use of $\gamma$ is based on the ``Gamma Trick'' \cite[Lemma~7.1.3]{SW05}
first introduced by Morgan and Sommese \cite{MS87}.

Second, since there exist many suitable methods to compute the start points $S$,
the following algorithm does not directly specify which one to utilize.
Nonetheless, to improve efficiency in this computation, the method
should, in some way, utilize the natural 2-homogeneous structure.

\begin{description}
  \item[Procedure \hbox{\rm $[v,R]= $} \Real\hbox{$(f,\sW,z,\gamma,y,\alpha)$}]
  \item[Input] A polynomial system $f:\bR^N\rightarrow\bR^{N-d}$, a witness set $\sW$ for a pure $d$-dimensional algebraic set $V\subset\Var(f)$,
      $z\in\bR^{N-d}$, $\gamma\in\bC$, $y\in\bR^N\setminus\Var_\bR(f)$, and $\alpha\in\bC^{N-d+1}$ such
      that Item~\ref{item:Main2V} of Theorem~\ref{thm:MainThmV} holds.
  \item[Output] A boolean $v$ which is $true$ if Items~\ref{item:Main1V},~\ref{item:Main3V}, and~\ref{item:Main4V} in
    Theorem~\ref{thm:MainThmV} have been computationally verified, otherwise $false$. If $v$ is $true$, $R$ is a finite subset of $\bR^N$
    containing a point on each connected component of the real algebraic set $\Var_\bR(f)$ contained in $V$.
  \item[Begin]\
\begin{enumerate}
 \item Construct the homotopy $H$ defined in (\ref{eq:HomotopyV}).
 \item\label{Item:2} Compute the solutions $S$ of $H(x,\lambda,1) = 0$.
 \begin{enumerate}
   \item Use $S$ to verify that Item~\ref{item:Main1V} of Theorem~\ref{thm:MainThmV} holds.  If it does not hold, {\bf Return}~$[false,\emptyset]$.
 \end{enumerate}
 \item\label{Item:3} Track the solution paths of $H$ starting at each point in $S$ to compute the sets $E$ and $E_1$ defined in Theorem~\ref{thm:MainThmV}.
 \begin{enumerate}
   \item If the tracking fails for a path or $\pi(E) \neq E_1$ where $\pi(x,\lambda) = x$, {\bf Return}~$[false,\emptyset]$.
 \end{enumerate}
 \item\label{Item:4} Use the homotopy membership test to compute the set $R$ consisting of the points in $E_1\cap\bR^N$ contained in $V$.
\end{enumerate}
  \item[Return \hbox{\rm $[true,R]$.}]
\end{description}

Since the endpoints $E$ computed in Step~\ref{Item:3} may be singular solutions of $H(x,\lambda,0) = 0$,
the use of an endgame, e.g., \cite{Parallel,PolyEndgame,PowerSeries,CauchyPoly},
together with adaptive precision tracking \cite{AMP3,AMP1,AMP2} may be required
to accurately compute them.  Also, Step~\ref{Item:3} should use the method of \cite{TransformInf}
to avoid infinite length paths.

\begin{example}\label{ex:Hypersurface}
To illustrate the algorithm for a hypersurface,
consider the polynomial $f(x_1,x_2,x_3) = (x_1+x_3)^2 + x_2^2$
with $V = \Var(f)$.  Clearly, $\Var_\bR(f) = \{(a,0,-a)~|~a\in\bR\}\subset\Sing(f)$.
Item~\ref{item:Main2V} holds with
$z = 1$, $\gamma = 2+3i$,
$$y = \left[\begin{array}{c} 3/8 \\ 5/9 \\ 1/3 \end{array}\right] \hbox{~~and~~} \alpha = \left[\begin{array}{c} 1/2-i/5 \\ 6/7+2i/3 \end{array}\right],
\hbox{~~where~~} i = \sqrt{-1}.$$
Let $H$ be the homotopy defined by (\ref{eq:HomotopyV}).
\begin{itemize}
\item For Step~\ref{Item:2}, we used a standard 2-homogeneous homotopy,
which required tracking $K(3,2) = 6$ paths, to compute the set $S$ consisting of the four nonsingular solutions of $H(x,\lambda,1) = 0$.
\item The four paths tracked in Step~\ref{Item:3}, which started at the points in $S$,
all converged with the endpoints of the two paths ending at the real point coinciding.
In particular, $E$ and $E_1$ both consist of three points with $\pi(E) = E_1$ where
$$E_1 = \left\{\begin{array}{c} (1/48, 0, -1/48), (-1/3+5i/9, 10/9+17i/24, -3/8+5i/9), \\
                                (-1/3-5i/9, 10/9-17i/24, -3/8-5i/9) \end{array}\right\}.$$
\item Since $V = \Var(f)$, we have $R = E_1\cap\bR^N = \{(1/48,0,-1/48)\}$.
\end{itemize}
It is easy to verify that the point $(1/48, 0, -1/48)$
is the minimizer of the distance between the point $y$ and $\Var_\bR(f)$,
as shown in Figure~\ref{Fig:Hypersurf}.

\begin{figure}[ht]
\centering
\includegraphics[width=5in,height=2in,angle=0]{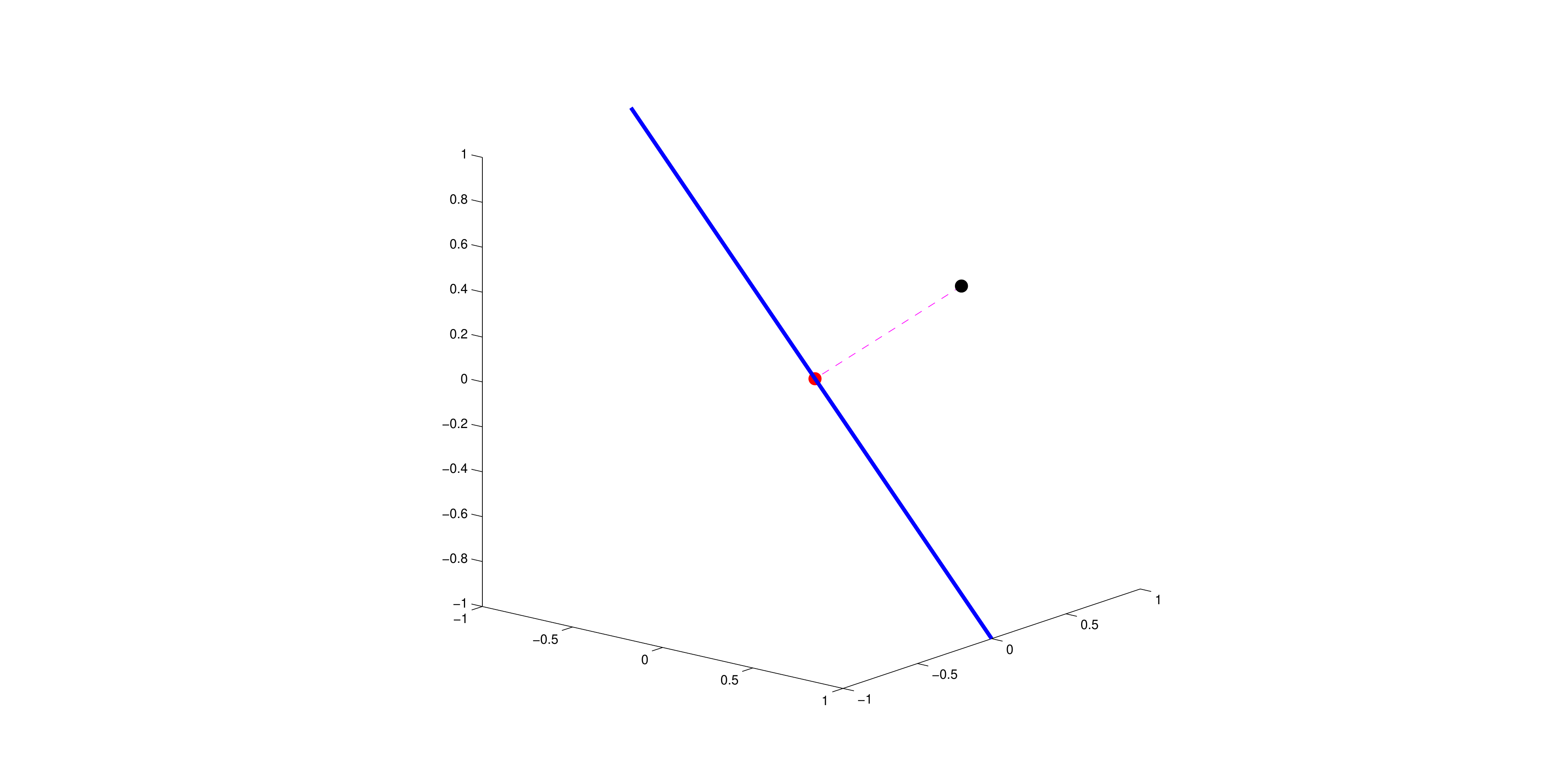}
\caption{Plot of $\Var_\bR(f)$ and the point minimizing the distance between $y$ and $\Var_\bR(f)$}
\label{Fig:Hypersurf}
\end{figure}
\end{example}

\begin{example}\label{ex:AlgSet}
To illustrate the algorithm for an algebraic set, consider the polynomial system
$$f(x) = \left[\begin{array}{c} g_1(x) + r_1 g_3(x) \\ g_2(x) + r_2 g_3(x) \end{array}\right]
\hbox{~~where~~}g(x_1,x_2,x_3) =
\left[\begin{array}{c} (x_1^2 - x_2)(x_1^2 + x_2^2 + x_3^2-1)(x_1 - 1) \\
(x_1x_2 - x_3)(x_1^2 + x_2^2 + x_3^2-1)(x_2 - 2) \\
(x_1x_3 - x_2^2)(x_1^2 + x_2^2 + x_3^2-1)(x_3 - 3) \end{array}\right]$$
with $r_1 = 1/3$ and $r_2 = 1/7$.
We want to investigate real points of the 
cubic curve $V = \{(x_1,x_1^2,x_1^3)~|~x_1\in\bC\}\subset\Var(f)$
which is input into \Real~via a witness set $\sW$.
Item~\ref{item:Main2V} holds with
$$z = \left[\begin{array}{c} 1/5 \\ 1/9 \end{array}\right],~~\gamma = 3/11-i/13,~~
y = \left[\begin{array}{c} 1/4 \\ 1/6 \\ -3/2 \end{array}\right], \hbox{~~and~~}
\alpha = \left[\begin{array}{c} 1/3-i/7 \\ 6/11+3i/4 \\ 2/3-7i/8 \end{array}\right]
\hbox{~~where~} i = \sqrt{-1}.$$

Let $H$ be the homotopy defined by (\ref{eq:HomotopyV}).
\begin{itemize}
\item We used a standard 2-homogeneous homotopy, which required tracking $300$ paths, 
    to compute the set $S$ in Step~\ref{Item:2} consisting of the $95$ nonsingular solutions of $H(x,\lambda,1) = 0$.
\item All 95 of the paths tracked in Step~\ref{Item:3} starting from the points in $S$ converged
    with the set $E_1\cap\bR^N$ consisting of $15$ points.
\item The homotopy membership test yields that $7$ of these $15$ points lie on $V$.  
\end{itemize}
The point of minimum distance on $V\cap\bR^N$ to $y$ is 
approximately $(0.168, 0.028, 0.005)$, which is displayed
in Figure~\ref{Fig:Cubic} with the other $6$ points on $V$.

\begin{figure}[ht]
\centering
\includegraphics[width=2.5in,angle=0,scale=2.5]{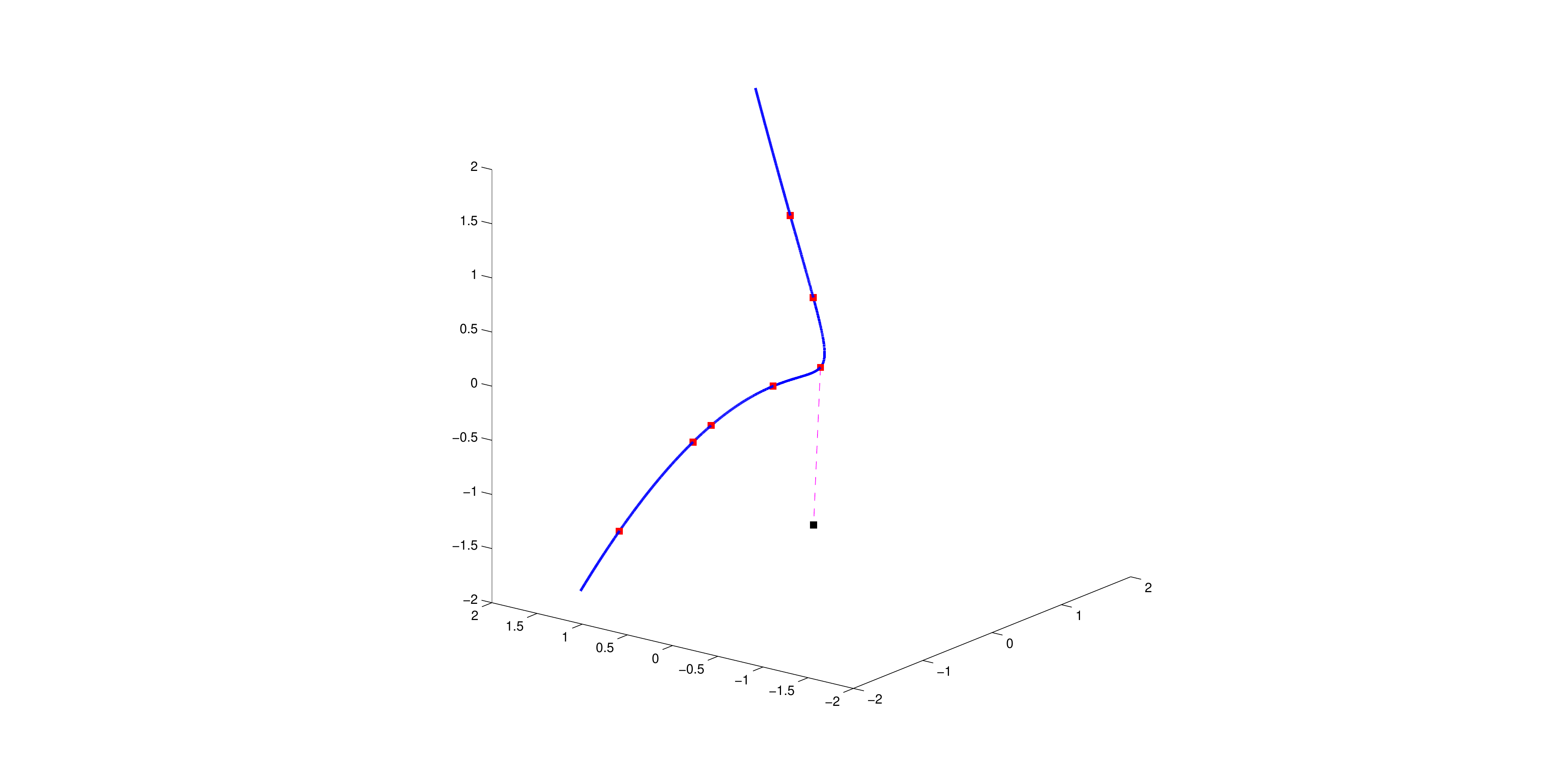}
\caption{Plot of $V\cap\bR^N$ and the point minimizing the distance between $y$ and $V\cap\bR^N$}
\label{Fig:Cubic}
\end{figure}
\end{example}

\section{Examples}\label{Sec:Examples}

The following examples were run using the software package Bertini v1.3.1 \cite{BHSW06}
on a server having four 2.3 GHz Opteron 6176 processors and 64 GB of memory
that runs 64-bit Linux.  The serial examples used one core while the parallel
examples used one manager and 47 working cores.
For the nonsingular solutions, we utilized
alphaCertified \cite{alphaCertifiedPaper,alphaCertified}
to certify reality.  For the singular solutions, we
determined reality based upon the size of the imaginary parts
using two different numerical approximations of the point.

\subsection{Hypersurface example}\label{Ex:Hypersurface}

Consider the polynomial provided in Example~5 of \cite{Real}, namely
$$
\begin{array}{ccl}
f(u_2,u_3,u_4,u_5) &=& 110u_5^2u_4u_3+190u_5u_4^2u_3+80u_4^3u_3+80u_5^2u_3^2+270u_5u_4u_3^2+160u_4^2u_3^2 \\
                   & & +80u_5u_3^3+80u_4u_3^3-32u_4u_3^2u_2-32u_3^3u_2-80u_5^2u_2^2-128u_5u_4u_2^2\\
                   & & -160u_5u_3u_2^2-112u_4u_3u_2^2-64u_3^2u_2^2-80u_5u_2^3-32u_3u_2^3+60u_5^2u_4\\
                   & & +220u_5u_4^2+160u_4^3+67u_5u_4u_3+136u_4^2u_3-24u_5u_3^2-88u_4u_3^2-64u_3^3\\
                   & & -100u_5^2u_2+32u_5u_4u_2+96u_4^2u_2-228u_5u_3u_2-108u_4u_3u_2-120u_3^2u_2\\
                   & & +20u_5u_2^2+96u_4u_2^2-56u_3u_2^2+110u_5u_4+80u_4^2+48u_4u_3-32u_3^2\\
                   & & +30u_5u_2+48u_4u_2-20u_3u_2. \end{array}
$$
The approach of \cite{Real} computes 26 real points on the hypersurface
which contains at least one point on each connected component of $\Var_\bR(f)$
using $y = 0$.  Since $y = 0$ does not satisfy the hypotheses of 
Theorem~\ref{thm:MainThmV}, we used $y = \left[\begin{array}{cccc} 4/3 & -9/5 & -5/7 & 8/9\end{array}\right]^T$
to compute at least one point on each connected component of $\Var_\bR(f)$.
In particular, we used serial processing with \Real~taking $V = \Var(f)$,
$z = 1$, and $\gamma\in\bC$ and $\alpha\in\bC^2$ to be random of unit length.

Let $H$ be the homotopy defined by (\ref{eq:HomotopyV}).
\begin{itemize}
\item For Step~\ref{Item:2}, we used a 2-homogeneous regeneration \cite{Regen}
    to compute the set $S$ consisting of the 151 nonsingular solutions of $H(x,\lambda,1) = 0$ in $10$ seconds.
\item For Step~\ref{Item:3}, each of the 151 paths converged with the set $E_1\cap\bR^N$ consisting
    of $28$ distinct points, which was computed in $4$ seconds.  
\item Since $V = \Var(f)$, $R = E_1\cap\bR^N$ which consists of $28$ points.
\end{itemize}

We note that since $|\Sing(f)| < \infty$, we could directly compute $\Var(H(x,\lambda,0))$
using a standard 2-homogeneous homotopy, which requires the tracking of $432$ paths.  
Bertini performed this computation in serial in 27 seconds which yielded
the same set $R$ of $28$ real critical points, as required by theory \cite{CoeffParam}.

\subsection{An example from filter banks}\label{Ex:FilterBanks}

Consider the polynomial system named F633 \cite{Wavelets}
that was considered in \cite{RealSolving}, which is available at \cite{GINV}.
This polynomial system consists of 9 polynomials in 10 variables.
Since two of the polynomials are linear and linearly independent,
we utilized intrinsic coordinates to reduce the number of variables
to 8 and the number of polynomials to 7, all of which are bilinear.
Since these 7 polynomials are not independent, we further reduced down
to a system of 6 bilinear polynomials in 8 variables, namely 
$$f(u_3,\dots,u_6,U_3,\dots,U_6) = 
\left[\begin{array}{c}
g(u_3,\dots,u_6,U_3,\dots,U_6) \\
g(U_3,\dots,U_6,u_3,\dots,u_6) \\
u_3 U_3 - 1 \\
u_4 U_4 - 1 \\
u_5 U_5 - 1 \\
u_6 U_6 - 1
\end{array}\right]$$
where
$$g(x_1,\dots,x_4,y_1,\dots,y_4) = 8(x_1y_2 + x_1y_3 + x_2y_3 + x_1y_4 + x_2y_4 + x_3y_4) + 4(x_1 + x_2 + x_3 + x_4) + 13.$$
The algebraic set $\Var(f)$ is an irreducible surface of degree $32$.
We used \Real~to compute a set of points containing a point from each connected component of $\Var_\bR(f)$ by taking
$$y = \left[\begin{array}{cccccccc} 1/5 & -3/4 & -2/3 & 7/9 & -4/7 & 12/13 & 1/2 & -10/11 \end{array}\right]^T,$$
$z\in\bR^6$, $\gamma\in\bC$, and $\alpha\in\bC^7$ to be random of unit length, and $\sW$ a witness set for $V = \Var(f)$.
Let $H$ be the homotopy defined by (\ref{eq:HomotopyV}).
\begin{itemize}
\item For Step~\ref{Item:2}, we used a standard 2-homogeneous homotopy, which required tracking $1792$ paths, 
    to compute the set $S$ consisting of the 274 nonsingular solutions of $H(x,\lambda,1) = 0$.  
    This computation took $120$ seconds in serial ($4$ seconds in parallel).
\item For Step~\ref{Item:3}, each of the 274 paths converged with the set $E_1\cap\bR^N$ consisting
    of $36$ distinct points.  This computation took one second in serial.
\item Since $V = \Var(f)$, $R = E_1\cap\bR^N$ which consists of $36$ points.
\end{itemize}

In Step~\ref{Item:2}, we could have used a 3-homogeneous homotopy since
the system itself is naturally 2-homogeneous.  However, this would increase the 
number of paths from $1792$ to $1960$.  Also, since $\Sing(f) = \emptyset$,
we could directly compute $\Var(H(x,\lambda,0))$ using a standard 2-homogeneous homotopy, 
which requires the tracking of $1792$ paths.
Bertini performed this computation in serial in $120$ seconds
yielding the same set $R$ of 36 real critical points, as required by theory \cite{CoeffParam}.

\subsection{A cubic-centered 12-bar linkage}\label{Ex:12bar}

Consider the 12-bar spherical linkage obtained by
locking the scissors of the collapsible cube
with 12 scissors linkages
presented in \cite{WamplerLarsonErdman:NewMobility:DETC:2007},
which is displayed in Figure~3 of \cite{MechanismMobility}.
Following the setup in \cite{MechanismMobility}, we will consider
the cube with side length 2 where we fix the
center at the origin and two adjacent vertices, say
$P_7 = (-1,1,-1)$ and $P_8 = (-1,-1,-1)$.
Let $P_1,\dots,P_6$ denote the position of the other
6 vertices yielding 18 variables.
The constraints on these vertices is that they must
maintain their initial relative distances yielding
a polynomial system $f$ consisting of the following $17$ polynomials:
\begin{align*}\label{eq:Cubic12bar}
    g_{ij} &= |P_i-P_j|^2 - 4,\nonumber\\
           &\qquad \{i,j\}\in\{(1,2),(3,4),(5,6),(1,5),(2,6),(3,7),(4,8),(1,3),(2,4),(5,7),(6,8)\};\\
    h_i &= |P_i|^2 - 3, \qquad i\in\{1,2,3,4,5,6\}.
\end{align*}
The algebraic set $\Var(f)$ consists of $8$, $34$, and $2$ irreducible components of 
dimension $1$, $2$, and $3$, respectively.  
Table~\ref{Tab:12bar} presents the degrees of these components.  
Let $V$ be the union of the one-dimensional irreducible components of $\Var(f)$,
which has degree $36$, and $\sW$ be a witness set for $V$.
Let $C_1,\dots,C_6$ denote the six irreducible curves of degree $4$ contained in $V$,
and $C_7$ and $C_8$ denote the two irreducible curves of degree $6$ contained in $V$.
The components $C_1,\dots,C_6$ are self-conjugate while $C_7$ and $C_8$ are conjugates of
each other.  That is, $C_7\cup C_8$ contains only finitely many real points 
which must be contained in $C_7\cap C_8$.  

\begin{table}
  \centering
  \begin{tabular}{|c|c|c|}
    \hline
    dimension & degree & \# components \\
    \hline
    3 & 8 & 2 \\
    \hline
    \multirow{6}{*}{2} & 4 & 2 \\
    \cline{2-3}
          & 8 & 14 \\
    \cline{2-3}
          & 12 & 12 \\
    \cline{2-3}
          & 16 & 1 \\
    \cline{2-3}
          & 20 & 4 \\
    \cline{2-3}
          & 24 & 1 \\
    \hline
    \multirow{2}{*}{1} & 4 & 6 \\
    \cline{2-3}
          & 6 & 2 \\
    \hline
  \end{tabular}
  \caption{Irreducible decomposition of $\Var(f)$}\label{Tab:12bar}
\end{table}

We used \Real~to compute a finite set of points containing a point on each connected component of $\Var_\bR(f)$ 
contained in $V$ by taking
$$
\begin{array}{rcl}
y & = & [0.142, 0.319, -0.286, -0.167, 0.276, 0.238, 0.217, -0.268, -0.089, \\ & & ~~~~~~~~~~~-0.198, 0.287, -0.042, -0.243, 0.119, 0.309, -0.312, 0.305, 0.162 ]^T,
\end{array}
$$
$z\in\bR^{17}$, $\gamma\in\bC$, and $\alpha\in\bC^{18}$ to be random of unit length.
Let $H$ be the homotopy defined by~(\ref{eq:HomotopyV}).
\begin{itemize}
\item For Step~\ref{Item:2}, we computed $S$ using a diagonal homotopy \cite{SVW04}
    by computing $A\cap B$ where $A = \Var(f - \gamma z) \times \bC^{18}$ and
    $$B = \Var(\lambda_0(x-y) + \lambda_1 \grad f_1(x)^T + \cdots + \lambda_{17} \grad f_{17}(x)^T,
    \alpha_0 \lambda_0 + \cdots + \alpha_{17} \lambda_{17} - 1).$$
    Since $\Var(f - \gamma z)$ is a curve of degree 480 and $B_{\sL} = B\cap\left(\sL \times \bC^{18}\right)$,
    where $\sL$ is a random line in $\bC^{18}$, consists of 13 points, the diagonal homotopy required tracking
    $480\cdot 13 = 6240$ paths, which yielded the $1536$ points in $S$.
    A witness set for $\Var(f-\gamma z)$ was computed using regeneration \cite{Regen} and $B_{\sL}$
    was computed using a standard 2-homogeneous homotopy.  Overall, this computation took $5.5$ minutes in parallel.
\item For Step~\ref{Item:3}, only $1440$ of the $1536$ paths converged and $\pi(E) = E_1$.  This computation took $18.5$ minutes in serial 
    ($26$ seconds in parallel) and found that the set $E_1\cap\bR^N$ consists of $283$ distinct points.
\item For Step~\ref{Item:4}, the homotopy membership test found that $R = V\cap E_1\cap\bR^N$ consists of $24$ points, 
    which took $80$ seconds in serial.
\end{itemize}

The set $R\setminus\Sing(f)$ consists of $16$ points and meets $C_i$ for $i = 1,\dots,6$.
This yields that $C_i\cap\bR^{18}$ is also one dimensional for $i = 1,\dots,6$.
Additionally, two points of $R$ lie in $C_7\cap C_8$,
one of which is presented in Figure~3 of \cite{MechanismMobility}.
Each of the other six points of $R$, which arose from 30 homotopy paths in 
Step~\ref{Item:4}, lies in the intersection of $V$
with some higher-dimensional components of $\Var(f)$.

\section{Conclusion}\label{Sec:Conclusion}

Infinitesimal deformations are widely used in real algebraic geometric algorithms.
By utilizing homotopy continuation to model the deformation, we have demonstrated
that one can obtain an algorithm for computing a finite set of real roots of a 
polynomial system containing a point on each connected component.  In particular, 
this algorithm computes a finite superset of the isolated roots over the real numbers.
This is similar to basic homotopy continuation in that one computes a
finite superset of the isolated roots over the complex numbers.
The isolated complex roots can be identified by, for example,
using the local dimension test of \cite{LocalDim}, but a similar test currently
does not exist over the real numbers.
Nonetheless, since many of the algorithms in numerical algebraic geometry depend
only on the ability to compute a superset of the isolated roots, we will investigate
what other computations can be performed in numerical {\em real} algebraic geometry
building from the algorithm presented here.

\section*{Acknowledgments}
The author would like to thank Mohab Safey El Din,
Charles Wampler, and the anonymous referee for their
helpful comments as well as the Institut Mittag-Leffler (Djursholm, Sweden)
for support and hospitality when working on this article.

\pdfbookmark[0]{References}{thebibliography}


\begin{thebibliography}{999}

\bibitem{RealSolving} P. Aubry, F. Rouillier, and M. Safey El Din.
\newblock Real solving for positive dimensional systems.
\newblock {\em J. Symbolic Comput.}, 34 (6), 543--560, 2002.

\bibitem{BGHM} B. Bank, M. Giusti, J. Heintz, and G. MBakop.
\newblock Polar varieties and efficient elimination.
\newblock {\em Math. Z.}, 238, 115--144, 2001.

\bibitem{BGHSS} B. Bank, M. Giusti, J. Heintz, M. Safey El Din, and E. Schost.
\newblock On the geometry of polar varieties.
\newblock {\em Appl. Algebra Engrg. Comm. Comput.}, 21, 33--83, 2010.
		
\bibitem{Basu} S. Basu, R. Pollack, and M.-F. Roy.
\newblock {\em Algorithms in real algebraic geometry}, volume~10 of {\em
  Algorithms and Computation in Mathematics}.
\newblock Springer-Verlag, Berlin, second edition, 2006.

\bibitem{BPR} S. Basu, R. Pollack, and M.-F. Roy.
\newblock On the combinatorial and algebraic complexity of quantifier elimination.
\newblock {\em J. ACM}, 43(6), 1002--1045, 1996.

\bibitem{LocalDim} D.J. Bates, J.D. Hauenstein, C. Peterson, and A.J. Sommese.
\newblock A numerical local dimension test for points on the
solution set of a system of polynomial equations.
\newblock {\em SIAM J. Numer. Anal.}, 47(5), 3608--3623, 2009.

\bibitem{AMP3} D.J. Bates, J.D. Hauenstein, and A.J. Sommese.
\newblock Efficient path tracking methods.
\newblock {\em Numer. Algorithms}, 58(4), 451--459, 2011.

\bibitem{Parallel} D.J. Bates, J.D. Hauenstein, and A.J. Sommese.
\newblock A parallel endgame.
\newblock {\em Contemp. Math.}, 556, 25--35, 2011.

\bibitem{AMP1} D.J. Bates, J.D. Hauenstein, A.J. Sommese, and C.W. Wampler.
\newblock Adaptive multiprecision path tracking.
\newblock {\em SIAM J. Numer. Anal.}, 46(2), 722--746, 2008.

\bibitem{BHSW06} D.J. Bates, J.D. Hauenstein, A.J. Sommese, and C.W. Wampler.
\newblock Bertini: Software for Numerical Algebraic Geometry.
\newblock Available at \url{http://www.nd.edu/~sommese/bertini}.

\bibitem{AMP2} D.J. Bates, J.D. Hauenstein, A.J. Sommese, and C.W. Wampler.
\newblock Stepsize control for adpative multiprecision path tracking.
\newblock {\em Contemp. Math.}, 496, 21--31, 2009.

\bibitem{Fewnomials} D.J. Bates and F. Sottile.
\newblock Khovanskii-Rolle continuation for real solutions.
\newblock {\em Found. Comput. Math.}, 11, 563--587, 2011.

\bibitem{RealSurface} G.M. Besana, S. Di Rocco, J.D. Hauenstein, A.J. Sommese, and C.W. Wampler.
\newblock Cell decomposition of almost smooth real algebraic surfaces.
\newblock Preprint, 2011.
\newblock Available at \url{http://math.tamu.edu/~jhauenst/preprints}.

\bibitem{BCR} J. Bochnak, M. Coste, and M.-F. Roy.
\newblock {\em Real algebraic geometry}, volume~36 of {\em Ergebnisse der
  Mathematik und ihrer Grenzgebiete (3) [Results in Mathematics and Related
  Areas (3)]}.
\newblock Springer-Verlag, Berlin, 1998.
\newblock Translated from the 1987 French original, Revised by the authors.

\bibitem{RealPositive} D. Cartwright.
\newblock An iterative method converging to a positive solution of certain systems of polynomial equations..
\newblock {\em J. Alg. Stat.}, 2, 1--13, 2011.

\bibitem{Canny} J. Canny.
\newblock Computing roadmaps of general semi-algebraic sets.
\newblock {\em Comput. J.}, 36(5), 504--514, 1993.

\bibitem{Collins} G.E. Collins.
\newblock Quantifier elimination for real closed fields by cylindrical algebraic decomposition.
\newblock Volume 33 of {\em Springer Lecture Notes in Computer Science}, 515--532, 1975.

\bibitem{CLO} D. Cox, J. Little, and D. O'Shea.
\newblock {\em Ideals, varieties, and algorithms}, third edition.
\newblock Springer, New York, 2007.

\bibitem{Wavelets} J.-C. Faug{\`e}re, F. Moreau de Saint-Martin, and F. Rouillier.
\newblock Design of regular nonseparable bidimensional wavelets using Gröbner basis techniques.
\newblock {\em IEEE Trans. Signal Process}, 46(4), 845--856, 1998.

\bibitem{GINV} V.P. Gerdt, Y.A. Blinkov, and D.A.Yanovich.
\newblock GINV project.
\newblock Available at \url{http://invo.jinr.ru/ginv/}.

\bibitem{GV88} D. Grigor'ev and N. Vorobjov.
\newblock Solving systems of polynomial inequalities in subexponential time.
\newblock {\em J. Symbolic Comput.}, 5, 37--64, 1988.

\bibitem{GV92} D. Grigor'ev and N. Vorobjov.
\newblock Counting connected components of a semialgebraic set in subexponential time.
\newblock {\em Comput. Complexity}, 2(2), 133--186, 1992.

\bibitem{alphaCertifiedPaper} J.D. Hauenstein and F. Sottile.
\newblock alphaCertified: certifying solutions to polynomial systems.
\newblock To appear in {\em ACM T. Math. Software}.

\bibitem{alphaCertified} J.D. Hauenstein and F. Sottile.
\newblock alphaCertified: software for certifying solutions to polynomial systems.
\newblock Available at \url{http://www.math.tamu.edu/~sottile/research/stories/alphaCertified}.

\bibitem{Regen} J.D. Hauenstein, A.J. Sommese, and C.W. Wampler.
\newblock Regeneration homotopies for solving systems of polynomials.
\newblock {\em Math. Comp.}, 80, 345--377, 2011.

\bibitem{RegenCascade} J.D. Hauenstein, A.J. Sommese, and C.W. Wampler.
\newblock Regenerative cascade homotopies for solving polynomial systems.
\newblock {\em Appl. Math. Comput.}, 218(4), 1240--1246, 2011.

\bibitem{HRS} J. Heintz, M.-F. Roy, P. Solern\'o.
\newblock Description of the connected components of a semialgebraic set in single exponential time.
\newblock {\em Discrete Comput. Geom.}, 11(2), 121--140, 1994.

\bibitem{PolyEndgame} B. Huber and J. Verschelde.
\newblock Polyhedral end games for polynomial continuation.
\newblock {\em Numer. Algorithms}, 18(1), 91--108, 1998.

\bibitem{John} F. John.
\newblock Extremum problems with inequalities as subsidiary conditions.
\newblock {\em Studies and {E}ssays {P}resented to {R}. {C}ourant on his 60th
              {B}irthday, {J}anuary 8, 1948},
\newblock pages 187--204, Interscience Publishers, Inc., New York, 1948.

\bibitem{RealRadical} J.B. Lasserre, M. Laurent and P. Rostalski.
\newblock A prolongation-projection algorithm for computing the finite real variety of an ideal.
\newblock {\em Theoret. Comput. Sci.}, 410(27--29), 2685–-2700, 2009.

\bibitem{NumRealRadicalSummary} J.B. Lasserre, M. Laurent, and P. Rostalski.
\newblock Semidefinite characterization and computation of zero-dimensional real radical ideals.
\newblock {\em Found. Comput. Math.}, 8(5), 607--647, 2008.

\bibitem{RealCurve}  Y. Lu, D.J. Bates, A.J. Sommese, and C.W. Wampler.
\newblock Finding all real points of a complex curve.
\newblock {\it Contemp. Math.}, 448, 183--205, 2007.

\bibitem{TransformInf} A.P. Morgan.
\newblock A transformation to avoid solutions at infinity for polynomial systems.
\newblock {\em Appl. Math. Comput.}, 18(1), 77--86, 1986.

\bibitem{MS87} A.P. Morgan and A.J. Sommese.
\newblock A homotopy for solving general polynomial systems that respects $m$-homogeneous structures.
\newblock {\em Appl. Math. Comput.}, 24(2), 101--113, 1987.

\bibitem{CoeffParam} A.P. Morgan and A.J. Sommese.
\newblock Coefficient-parameter polynomial continuation.
\newblock {\em Appl. Math. Comput.}, 29(2), 123--160, 1989.
\newblock Errata: {\em Appl. Math. Comput.}, 51, 207, 1992.

\bibitem{CauchyPoly} A.P. Morgan, A.J. Sommese, and C.W. Wampler.
\newblock Computing singular solutions to polynomial systems.
\newblock {\em Adv. in Appl. Math.}, 13(3), 305--327, 1992.

\bibitem{PowerSeries} A.P. Morgan, A.J. Sommese, and C.W. Wampler.
\newblock A power series method for computing singular solutions to nonlinear analytic systems.
\newblock {\em Numer. Math.}, 63(3), 391--409, 1992.

\bibitem{PolyProduct} A.P. Morgan, A.J. Sommese, and C.W. Wampler.
\newblock A product-decomposition bound for Bezout numbers.
\newblock {\em SIAM J. Numer. Anal.}, 32(4), 1308--1325, 1995.

\bibitem{R92} J. Renegar.
\newblock On the computational complexity and geometry of the first order theory of the reals.
\newblock {J. Symbolic Comput.}, 13(3), 255--352.

\bibitem{Real} F. Rouillier, M.-F. Roy, and M. Safey El Din.
\newblock Finding at least one point in each connected component of a real algebraic set defined by a single equation.
\newblock {\em J. Complexity}, 16 (4), 716--750, 2000.

\bibitem{Seidenberg} A. Seidenberg.
\newblock A new decision method for elementary algebra.
\newblock {\em Ann. of Math. (2)}, 60, 365--374, 1954.

\bibitem{Cascade} A.J. Sommese and J. Verschelde.
\newblock Numerical homotopies to compute generic points on positive dimensional algebraic sets.
\newblock {\it J. Complexity}, 16(3), 572--602, 2000.
\newblock Complexity theory, real machines, and homotopy (Oxford, 1999).

\bibitem{SVW04} A.J. Sommese, J. Verschelde, and C.W. Wampler.
\newblock Homotopies for intersecting solution components of polynomial systems.
\newblock {\em SIAM J. Numer. Anal.} 42(4), 1552--1571, 2004.

\bibitem{SVW01a} A.J. Sommese, J. Verschelde, and C.W. Wampler.
\newblock Numerical decomposition of the solution sets of polynomial systems into irreducible components.
\newblock {\em SIAM J. Numer. Anal.}, 38(6), 2022--2046, 2001.

\bibitem{SVW01b} A.J. Sommese, J. Verschelde, and C.W. Wampler.
\newblock Numerical irreducible decomposition using projections from points on the components.
\newblock In {\em Symbolic computation: solving equations in algebra, geometry,
  and engineering (South Hadley, MA, 2000)}, volume 286 of {\em Contemp. Math.},
37--51. Amer. Math. Soc., Providence, RI, 2001.

\bibitem{SVW02b} A.J. Sommese, J. Verschelde, and C.W. Wampler.
\newblock Symmetric functions applied to decomposing solution sets of polynomial systems.
\newblock {\em SIAM J. Numer. Anal.}, 40(6), 2026--2046, 2002.

\bibitem{SVW01c} A.J. Sommese, J. Verschelde, and C.W. Wampler.
\newblock Using monodromy to decompose solution sets of polynomial systems into irreducible components.
\newblock In {\em Applications of algebraic geometry to coding theory, physics
  and computation ({E}ilat, 2001)}, volume~36 of {\em NATO Sci. Ser. II Math.
  Phys. Chem.}, 297--315. Kluwer Acad. Publ., Dordrecht, 2001.

\bibitem{NAG} A.J. Sommese and C.W. Wampler.
\newblock Numerical algebraic geometry.
\newblock {\it The mathematics of numerical analysis (Park City, UT, 1995)},
 749--763, {\it Lectures in Appl. Math.}, 32, Amer. Math. Soc., Providence, RI,  1996.

\bibitem{SW05} A.J. Sommese and C.W. Wampler.
\newblock {\em The Numerical solution of systems of
          polynomials arising in engineering and science.}
\newblock World Scientific Press, Singapore, 2005.

\bibitem{LinearProduct} J. Verschelde and R. Cools.
\newblock Symbolic homotopy construction.
\newblock {\it Appl. Algebra Engrg. Comm. Comput.}, 4(3), 169--183, 1993.

\bibitem{MechanismMobility} C.W. Wampler, J.D. Hauenstein, and A.J. Sommese.
\newblock Mechanism mobility and a local dimension test.
\newblock {\em Mech. Mach. Theory}, 46(9), 1193--1206, 2011.

\bibitem{WamplerLarsonErdman:NewMobility:DETC:2007} C.~Wampler, B.~Larson, and A.~Edrman.
\newblock A new mobility formula for spatial mechanisms.
\newblock In {\em Proc. DETC/Mechanisms \& Robotics Conf., Sept. 4--7, Las
  Vegas, NV (CDROM)}, 2007.

\end{thebibliography}
\end{document}